\documentclass[a4paper,3p,preprint,12pt]{elsarticle}

\usepackage{graphicx}
\usepackage{amsmath}
\usepackage{amsfonts}
\usepackage{subfig}
\usepackage{xspace}
\usepackage{url}
\usepackage{sidecap}
\usepackage[colorlinks=true,citecolor=blue,urlcolor=blue]{hyperref}
\usepackage{tikz}
\usetikzlibrary{calc}

\graphicspath{{figs/}}

\newtheorem{defn}{Definition}
\newtheorem{mytheorem}{Theorem}
\newtheorem{prob}{Problem}
\newtheorem{example}{Example}
\newproof{proof}{Proof}

\newcommand{\bigoh}[1]{O(#1)}
\newcommand{\diag}[1]{\mathrm{diag}(#1)}

\newcommand{\grad}{\nabla}

\newcommand{\dx}{\Delta x}

\newcommand{\Real}{\mathbb{R}}

\newcommand{\myeps}{\varepsilon}

\renewcommand{\v}[1]{\boldsymbol{#1}}
\newcommand{\m}[1]{\textup{\textsf#1}}

\newcommand{\vu}{\v{u}}

\newcommand{\code}[1]{\textsf{#1}}

\newcommand{\surf}{\mathcal{S}}
\newcommand{\gradsurf}{\grad_{\!\surf}}
\newcommand{\lap}{\Delta}
\newcommand{\lapsurf}{\lap_{\surf}}
\newcommand{\x}{\v{x}}
\newcommand{\cp}{\textup{cp}}

\newcommand{\lapsharp}{\lap^{\!\#}_{\myeps}}

\newcommand{\Matlab}{{\sc Matlab}\xspace}

\newcommand{\uband}{L}

\newcommand{\mDlap}{\mathbf{\Delta}_h}
\newcommand{\unstab}[1]{\widetilde{#1}}
\newcommand{\DEstab}{\m{M}}
\newcommand{\DEunst}{\unstab{\m{M}}}

\begin{document}
\begin{frontmatter}

\title{Solving eigenvalue problems on curved surfaces using the
  Closest Point Method}

\author[ox]{Colin B.~Macdonald\corref{cor2}\fnref{fn1}}
\ead{macdonald@maths.ox.ac.uk}

\author[cims]{Jeremy Brandman\fnref{fn2}}
\ead{brandman@cims.nyu.edu}

\author[sfu]{Steven J.~Ruuth\fnref{fn3}}
\ead{sruuth@sfu.ca}

\cortext[cor2]{Principal corresponding author}

\fntext[fn1]{The work of this author was supported by an NSERC postdoctoral
  fellowship, NSF grant number CCF-0321917, and by Award No
  KUK-C1-013-04 made by King Abdullah University of Science and
  Technology (KAUST).}

\fntext[fn2]{The work of this author was supported by an NSF
  Mathematical Sciences Postdoctoral Research Fellowship.}

\fntext[fn3]{The work of this author was partially supported by a
  grant from NSERC Canada.}

\address[ox]{Mathematical Institute, University of Oxford, OX1\,3LB, UK.}

\address[cims]{Department of Mathematics, Courant Institute of Mathematical
  Sciences, New York University.}

\address[sfu]{Department of Mathematics, Simon Fraser University, Burnaby,
  British Columbia, V5A\,1S6 Canada.}

\begin{abstract}
  Eigenvalue problems are fundamental to mathematics and science.  We
  present a simple algorithm for determining eigenvalues and
  eigenfunctions of the Laplace--Beltrami operator on rather general
  curved surfaces.  Our algorithm, which is based on the Closest Point
  Method, relies on an embedding of the surface in a
  higher-dimensional space, where standard Cartesian finite difference
  and interpolation schemes can be easily applied.  We show that there
  is a one-to-one correspondence between a problem defined in the
  embedding space and the original surface problem.  For open
  surfaces, we present a simple way to impose Dirichlet and Neumann
  boundary conditions while maintaining second-order accuracy.
  Convergence studies and a series of examples demonstrate the
  effectiveness and generality of our approach.
\end{abstract}


\end{frontmatter}

\section{Introduction}
\label{sec:intro}

The study of eigenvalues and eigenfunctions of the Laplacian operator
has long been a subject of interest in mathematics, physics,
engineering, computer science and other disciplines.  Of
considerable importance is the case where the underlying domain is a
curved surface, $S$, in which case the problem becomes one of finding
eigenvalues and eigenfunctions of the corresponding Laplace--Beltrami
operator
\begin{equation}
-\gradsurf \cdot \gradsurf u = \lambda u, \label{eq:eig_eq}
\end{equation}
or, more generally, the elliptic operator
\[
-\gradsurf \cdot (a(x) \gradsurf  u) = \lambda u.
\]
The Laplace--Beltrami eigenvalue problem has played a prominent role
in recent years in data analysis.  For example, in \cite{Reuter2005},
eigenvalues of the Laplace--Beltrami operator were used to extract
``fingerprints'' which characterize surfaces and solid objects.  In
\cite{Belkin,Coifman2006}, Laplace--Beltrami eigenvalues and
eigenfunctions were used for dimensionality reduction and data
representation.  Other application areas include smoothing of surfaces
\cite{Seo} and the segmentation and registration of shape
\cite{Reuter2009}.

Analytical solutions to the Laplace--Beltrami eigenvalue problem are
rarely available, so it is crucial to be able to numerically
approximate them in an accurate and efficient manner.  Partial
differential equations on surfaces, including eigenvalue problems,
have traditionally been approximated using either (a) discretizations
based on a parameterization of the surface \cite{Glowinski}, (b)
finite element discretizations on a triangulation of the surface
\cite{Reuter:2006:shapeDNA}, or (c) embedding techniques which solve
some embedding PDE in a small region near the surface
\cite{Brandman:JSC2008:eigenvalues} (see also the related works
\cite{Bertalmio:2001, XuZhao:2003:PDEsMovIntf,
  GreerBertozzi:2006:biharmonic, Greer:2006, Nemitz:Submitted,
  Dziuk:2008, LeungLowengrubZhao:2011:PDEs_solving_surf}).

Parameterization methods (a) are often effective for simple surfaces
\cite{Glowinski}, but for more complicated geometries have the
deficiency of introducing distortions and singularities into the
method through the parameterization \cite{Floater05:surface_param}.
Approaches based on the finite element method can be deceptively
difficult to implement; as described in \cite{Reuter:2006:shapeDNA},
``even though this method seems to be very simple, it is quite tricky
to implement''.  Embedding methods (c) have gained a considerable
following because they permit PDEs on surfaces to be solving using
standard finite differences.

This paper proposes a simple and effective embedding method for the
Laplace--Beltrami eigenvalue problem based on the \emph{Closest Point
  Method}.  The Closest Point Method is a recent embedding method that
has been used to compute the numerical solution to a variety of
partial differential equations
\cite{Ruuth/Merriman:jcp08:CPM,cbm:lscpm,cbm:phd,cbm:icpm}, including
in-surface heat flow, reaction-diffusion equations, and higher-order
motions involving biharmonic and ``surface diffusion'' terms.  Unlike
traditional embedding methods, which are built around level set
representatives of the surface, the Closest Point Method is built
around a closest point representation of the surface.  This allows for
general smooth surfaces with boundaries and does not require the
surface to have an inside/outside \cite{Ruuth/Merriman:jcp08:CPM}.  In
addition, the method does not introduce artificial boundary conditions
at the edge of the computation band.  Such artificial boundary
conditions typically lead to low-order accuracy \cite{Greer:2006}.

Here we apply the Closest Point Method to the problem of determining
the eigenvalues and eigenmodes of the Laplace--Beltrami operator on a
surface.  We begin by demonstrating that, for closed surfaces, there
is a one-to-one correspondence between the eigenvalues of the
embedding problem and the original surface problem.  Later, we
consider open surfaces and present simple techniques for achieving
high-order accurate approximations to Dirichlet and homogeneous
Neumann boundary conditions.  Our proposed method retains the usual
advantages of the Closest Point Method, namely generality with respect
to the surface, high-order accuracy and simplicity of implementation.

The paper unfolds as follows.  Section~\ref{sec:cpm} provides key
background on the Closest Point Method.
Section~\ref{sec:embedded_ev_probs} proposes an embedding problem used
to solve the Laplace--Beltrami eigenvalue problem and explains why a
similar embedding problem leads to spurious eigenvalues.
Section~\ref{sec:discretization} provides discretization details.  In
Section~\ref{sec:bc}, a second-order discretization of boundary
conditions is described for open surfaces.
Section~\ref{sec:numerical_results} validates the method with a number
of convergence studies and examples on complex shapes.  Finally,
Section~\ref{sec:conclusions} gives a summary and conclusions.

\section{The Closest Point Method}
\label{sec:cpm}

We now review the ideas behind the Closest Point Method
\cite{Ruuth/Merriman:jcp08:CPM} which are relevant to the problem of
finding Laplace-Beltrami eigenvalues and eigenfunctions.

The representation of the underlying surface is fundamental to any
numerical method for PDEs on surfaces. The Closest Point Method relies
on a closest point representation of the underlying surface.
\begin{defn}[Closest point function]
  Given a surface $\surf$, $\cp(\x)$ refers to a (possibly non-unique)
  point belonging to $\surf$ which is closest to $\x$.
\end{defn}

The closest point function, defined in a neighborhood of a surface,
gives a representation of the surface.  This \emph{closest point
  representation} allows for general surfaces with boundaries and does
not require the surface to have an inside/outside.  The surface can be
of any codimension \cite{Ruuth/Merriman:jcp08:CPM}, or even of mixed
codimension \cite{cbm:icpm}.

The goal of the Closest Point Method is to replace a surface PDE by a
related PDE in the embedding space which can be solved using finite
difference, finite element or other standard methods.  In the case of
the Laplace-Beltrami eigenvalue problem, this approach relies on the
following result, which states that the Laplace--Beltrami operator
$\lapsurf$ may be replaced by the standard Laplacian $\lap$ in the
embedding space $\Real^d$ under certain conditions.

\begin{mytheorem} \label{cp_principle1} Let $\surf$ be a smooth closed
  surface in $\Real^d$ and $u: \surf \to \Real$ be a smooth function.
  Assume the closest point function $\cp(\x)$ is defined in a
  neighborhood $\Omega \subset \Real^d$ of $\surf$.  Then
  \begin{align}  \label{eq:LBcp_principle}
    \lapsurf u(\x) &= \lap(u(\cp(\x)))
    \quad \text{for $\x \in \surf$.}
  \end{align}
  Note that the right-hand side is well-defined because
  $u(\cp(\cdot))$ can be evaluated at points both on and off the
  surface.
\end{mytheorem}
This result follows from the principles in
\cite{Ruuth/Merriman:jcp08:CPM}.

Because the function $u(\cp(\x))$, known as the closest point
extension of $u$, is used throughout this paper, we make the following
definition.

\begin{defn}[Closest point extension]
  Let $\surf$ be a smooth surface in $\Real^d$.  The \emph{closest
    point extension} of a function $u: \surf \to \Real$ to a
  neighborhood $\Omega$ of $\surf$ is the function $v:\Omega \to
  \Real$ defined by
  \begin{align}
    v(\x) = u(\cp(\x)).
  \end{align}
\end{defn}

\section{The embedded eigenfunction problem}
\label{sec:embedded_ev_probs}

Our objective is to develop a simple, effective method for solving the
following surface eigenvalue problem:
\begin{prob}[Laplace--Beltrami eigenvalue problem]  \label{ev_prob}
   Given a surface $\surf$, determine the eigenfunctions $u :
  \surf \to \Real$ and eigenvalues $\lambda$ satisfying
  \begin{align}  \label{eq:LBev}
    -\lapsurf(u(\x)) = \lambda u(\x), \quad \text{for $\x \in \surf$.}
  \end{align}
\end{prob}
If the surface $\surf$ is open, then the problem will also have
boundary conditions: we address this in Section~\ref{sec:bc}.

In this section, we assume that a closest point representation of the
surface is available and consider two associated embedding problems.
We will see that the first, a direct extension of \eqref{eq:LBcp_principle}, leads to an ill-posed problem.  However a
straightforward modification of this problem leads to our second embedding problem,
which we show is equivalent to \eqref{eq:LBev}.

\subsection{A first try}
\label{sec:illposed_ev_prob}
We first consider the following embedded eigenvalue problem, which is
directly motivated by \eqref{eq:LBcp_principle}.
\begin{prob}[Ill-posed embedded eigenvalue problem] \label{prob:illposed_ev_prob}
  Determine the eigenfunctions $v : \Omega \subset \Real^d \to
  \Real$ and eigenvalues $\lambda$ satisfying
  \begin{align}  \label{eq:evCP}
    -\lap(v(\cp(\x))) = \lambda v(\x),
  \end{align}
  in a neighborhood $\Omega \subset \Real^d$ of $\surf$.
\end{prob}

Solutions to Problem~\ref{ev_prob} and
Problem~\ref{prob:illposed_ev_prob} are closely related.  Every
solution to the embedding problem, restricted to the surface, is a
solution to the surface problem.  Conversely, except for the $\lambda
= 0$ case, every surface eigenfunction corresponds to a unique
solution of Problem~\ref{prob:illposed_ev_prob}.  These results are
established in \ref{app:unst_case_thms}.

Notably, the one-to-one correspondence between solutions breaks down
for the $\lambda = 0$ case (the null-eigenspace).  Not only is this
case significant in theory, in practice it makes
Problem~\ref{prob:illposed_ev_prob} ill-posed.

\subsubsection{The null-eigenspace}
\label{sec:nulleigenspace}

The constant eigenfunction $u(\x) = c$ and $\lambda=0$ is a solution
to \eqref{eq:LBev}.  Now consider a function on $\Omega$ which agrees
with $u$ on the surface (but differs off the surface)
\begin{align*}
  \text{$v : \Omega \to \Real$,  such that $v(\x) = c$ for $\x \in \surf$,}
\end{align*}
and note that $v(\x)$ is a null-eigenfunction of \eqref{eq:evCP}.
Because $v(\x)$ is arbitrary for $\x \in \Omega \backslash \surf$, the
set of null-eigenfunctions for \eqref{eq:evCP} is much larger than the
set of null-eigenfunctions for \eqref{eq:LBev}.  In fact, the set of
null-eigenfunctions for \eqref{eq:evCP} is infinite-dimensional: any
(linearly independent) change off the surface gives a new linearly
independent eigenfunction.  For points off of the surface, these
null-eigenfunctions need not even be smooth.

This example demonstrates the essential flaw of
Problem~\ref{prob:illposed_ev_prob}: when $\lambda = 0$, only the
surface values of eigenfunctions are determined by \eqref{eq:evCP}
(i.e., eigenfunctions can take on arbitrary values elsewhere).  Not
surprisingly, the infinite-dimensionality of the null-eigenspace
causes problems for numerical methods based on approximating
\eqref{eq:evCP}. For example, a large number of eigenfunctions with
near-zero eigenvalues are observed in the numerical experiment of
Figure~\ref{fig:spectrum_hist_stab_unstab} in
Section~\ref{sec:stabCPM} below.

\subsection{The fix: a modified embedded eigenvalue problem}  \label{sec:thefix}

To avoid the null-eigenspace found in Problem 2,
we consider a modified embedded eigenvalue problem.  Our approach uses the split operator introduced in \cite{cbm:icpm}:
\begin{defn}[Operator $\lapsharp$] Given $\Omega \subset \Real^d$
  containing a surface $\surf$ and a function $v:\Omega \to \Real$,
  the operator $\lapsharp$ is defined as
  \begin{align}  \label{lapSharpDefn}
    \lapsharp v(\x) := \lap(v(\cp(\x)))
    - \frac{2d}{\myeps^2}
    \Big[ v(\x) - v(\cp(\x)) \Big].
  \end{align}
  where $0<\myeps \ll 1$.
\end{defn}
The factor $2d$ (twice the dimension of the embedding space $\Omega$)
is for later notational convenience.
We can view $\lapsharp v$ as $\lap(v(\cp))$ plus a penalty for large
change in the normal direction: specifically $\lapsharp v$ will be
large if $|v(\x) - v(\cp(\x))|$ is not $\bigoh{\myeps^2}$.  Using this
operator, we pose another embedded eigenvalue problem:
\begin{prob}[Regularized embedded eigenvalue problem]  \label{emb_ev_prob_fix}
 Determine all eigenfunctions $v : \Omega \subset \Real^d \to
  \Real$ and eigenvalues $\lambda$ satisfying
  \begin{align}  \label{eq:LBevstab}
    -\lapsharp v(\x) = \lambda v(\x),
  \end{align}
  in an embedding space $\Omega \subset \Real^d$ containing the
  surface $\surf$.
\end{prob}

For eigenvalues $\lambda < \frac{2d}{\myeps^2}$, we can show a
one-to-one correspondence between Problem~\ref{ev_prob} and
Problem~\ref{emb_ev_prob_fix}.

\begin{mytheorem}[Equivalence of two eigenvalue problems]
  \label{th:discLapSharpEquiv}
  Suppose $\surf$ is a smooth surface embedded in $\Real^d$ and that
  $\Omega \subset \Real^d$ is a neighborhood of the surface.
  Then, for every eigenfunction $u:C^2(\surf) \to \Real$ of \eqref{eq:LBev} with eigenvalue
  $\lambda < \frac{2d}{\myeps^2}$, there exists a unique
  eigenfunction $v : \Omega \to \Real$ of \eqref{eq:LBevstab} with eigenvalue $\lambda$ which agrees with
 $u$ on $\surf$.
  The eigenfunction $v$ is given by
  \begin{align}  \label{embeddedEigenfunc}
    v(\x) = \frac{\lap(u(\cp(\x))) +
      \frac{2d}{\myeps^2} u(\cp(\x))}{-\lambda + \frac{2d}{\myeps^2}}.
  \end{align}
  Conversely, for every eigenfunction $v(\x)$ of
  \eqref{eq:LBevstab} with eigenvalue $\lambda$, the restriction of $v(x)$ to $\surf$ is an eigenfunction of \eqref{eq:LBev} with
  eigenvalue $\lambda$.
\end{mytheorem}
\begin{proof}
Existence and uniqueness of $v$ follow directly from the condition that $v$ agrees with $u$ on $\surf$ and the definition of $\lapsharp v$.
The converse follows by choosing $x \in \surf$ in \eqref{eq:LBevstab} and applying Theorem~\ref{cp_principle1}.
\end{proof}
\paragraph{Remark 1}
Note that when $\lambda = 0$ in \eqref{eq:LBevstab}, $v$ is 
determined in $\Omega$ by its values on the surface,
avoiding the null-eigenspace problem of \eqref{eq:evCP}.

\paragraph{Remark 2}
If one tries to compute the eigenvalues $\lambda \ge
\frac{2d}{\myeps^2}$ of Problem~\ref{emb_ev_prob_fix}, one encounters
ill-posedness similar to the ill-posed embedded eigenvalue problem
\eqref{eq:evCP}.  This is confirmed by the numerical results in
Section~\ref{sec:high_freq_ews}.  The shift of the null-space by
$\frac{2d}{\myeps^2}$ is due to the additional term
$\frac{2d}{\myeps^2}\big( v(\x) - v(\cp(\x)) \big)$ in
Problem~\ref{emb_ev_prob_fix}.  As we shall see in the next section
and Section~\ref{sec:high_freq_ews}, this ill-posedness will not be
an issue for practical purposes since the parameter $\epsilon$ can
be chosen proportional to the grid spacing.

\begin{example}  \label{example_DEstab_circle}
Let $\surf$ be a circle of radius $R$ embedded in 2D and
consider the eigenfunction $u=\cos(\sqrt{\lambda}R\theta)$ with
eigenvalue $\lambda$.  Applying Theorem~\ref{th:discLapSharpEquiv}, we
get
\begin{align*}
  v(\x) = \left(
    \frac{\lambda \myeps^2 \left(1 - R^2/r^2\right)}{-\lambda \myeps^2 + 4} + 1
  \right) \cos(\sqrt{\lambda}R\theta),
\end{align*}
and we note that indeed $v(\x) = u(\x)$ on the surface.
\end{example}

\subsubsection{Discretizing the regularized operator: choice of $\myeps$}  \label{sec:choosing_eps}

In this work, we choose $\myeps = \Delta x$, where $\Delta x$ is the
underlying grid spacing.  In two dimensions, with this choice of
$\myeps$ and a centered five-point discretization of the Laplacian
operator, we begin discretizing the regularized operator
\eqref{lapSharpDefn} to obtain
\begin{multline}  \label{lapSharpDefn_discrete}
  \lapsharp u(x,y)  \approx \frac{1}{\dx^2}
  \big[ -4 u(x,y) +  u(\cp(x+\dx,y)) + u(\cp((x-\dx,y)) \,+ \\ u(\cp(x,y+\dx)) + u(\cp(x,y-\dx)) \big],
\end{multline}
and analogously in higher dimensions.  Note that by choosing $\myeps =
\Delta x$, the resulting diagonal term (i.e., the ``center'' of the
finite difference stencil) is without a closest point extension: where
diagonal terms appear in the discretization, consistency does not
require an interpolation to the closest point on the surface. See
\cite{cbm:icpm} for details.

While \eqref{lapSharpDefn_discrete} is an approximation to the
regularized operator \eqref{lapSharpDefn}, it is not completely
discrete since, for example, $\cp(x+\dx,y)$ is generally not a grid
point.  In Section~\ref{sec:stabCPM} we will complete this
discretization and solve the resulting eigenvalue problem.

\section{Numerical Discretization}  \label{sec:discretization}

In this section, we discretize $\lapsharp u$ following the approach
given in \cite{cbm:icpm}.  The eigenvalues of the resulting matrix,
computed using standard techniques, approximate the eigenvalues of
\eqref{eq:LBev}.

Following \cite{cbm:icpm}, our computational domain is a narrow band
of $m$ grid points $\uband = \{\x_1, \x_2, \ldots, \x_m\}$ enveloping
the surface $\surf$.  We approximate a function $u$ at all points in
$\uband$ as the vector $\vu \in \Real^m$ where $u_i \approx u(\x_i)$
for each $\x_i \in \uband$.

\subsection{Discrete closest point extensions}
To choose the list of points $\uband$, we consider the problem of
evaluating $u(\cp(\x_j))$ for a particular point $\x_j$.  As shown in
Figure~\ref{fig:cpext}, $\cp(\x_j)$ is not typically a grid point, so
we approximate the value of $u(\cp(\x_j))$ using interpolation of the
values of $u$ at neighboring grid points.  This interpolation has a
certain stencil associated with it \cite{cbm:icpm} and we choose
$\uband$ such that it contains the union of all the required
interpolation stencils (an algorithm for doing this is given in
\cite{cbm:icpm}).
\begin{SCfigure}[2]
  \begin{tikzpicture}[scale=0.5]
    \filldraw[blue!20!white] (0.85,0.85) rectangle (4.15,4.15);
    \foreach \p in {
            (1,0),(2,0),(3,0),(4,0),(5,0),(6,0),
            (1,1),(2,1),(3,1),(4,1),(5,1),(6,1),
      (0,2),(1,2),(2,2),(3,2),(4,2),(5,2),(6,2),
      (0,3),(1,3),(2,3),(3,3),(4,3),(5,3),(6,3),
      (0,4),(1,4),(2,4),(3,4),(4,4),(5,4),
      (0,5),(1,5),(2,5),(3,5),(4,5),(5,5),
      (0,6),(1,6),(2,6),(3,6)}
      \filldraw \p circle (1.5pt);

    \begin{scope}
      \clip (-0.3,-0.2) rectangle (4,4.8);
      \draw (-1.3,0.3) circle (4.5);
    \end{scope}

    \draw (5,4) node[below right=-2pt] {\small $\x$};

    \draw (-1.3,0.3)+(28:4.5) -- (5,4);
    \filldraw (-1.3,0.3)+(28:4.5) circle (5pt);

    \begin{scope}[label distance=-4pt]
      \node [label=180:{\small $\cp(\x)$}] at ($ (-1.3,0.3)+(28:4.5) $) {};
    \end{scope}
  \end{tikzpicture}
  \caption{A discrete closest point extension being applied to extend
    $u(\cp(\x))$ using degree $p=3$ (bicubic) interpolation.  The
    shaded region indicates the interpolation stencil for $\cp(\x)$.}
  \label{fig:cpext}
\end{SCfigure}
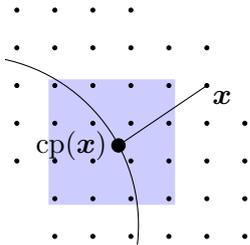

In this work, we use Barycentric Lagrange polynomial interpolation
\cite{Trefethen:barycentric04} which is linear in the values of $\vu$.
This allows us to express the closest point extension to any chosen set of
points as a multiplication by an \emph{extension matrix} $\m{E}$,
where each row of $\m{E}$ represents one discrete closest point extension.

\subsection{Discretizing the Laplace operator}

Combining the interpolation step, given by the matrix $\m{E}$, with a
standard, symmetric finite difference discretization of $\lap$ (e.g.,
the standard five-point stencil in 2D or the standard nine-point
stencil in 3D), yields a discretization of $\lap \left( u
  (\cp(\uband)) \right)$:

\begin{align*}
  \lap \left( u (\cp(\uband)) \right)  \approx
  \mDlap \m{E} \vu  =
  \raisebox{-0.98cm}{\includegraphics{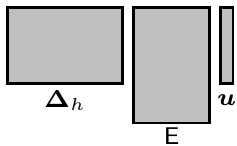}}
  =
  \raisebox{-0.38cm}{\includegraphics{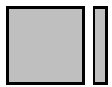}}
  = \DEunst \vu,
\end{align*}
where $\DEunst = \mDlap\m{E}$ is an $m \times m$ matrix intended to
approximate the Laplace--Beltrami operator.

However, the matrix $\DEunst$ is based on discretizing
$\lap(u(\cp(\cdot)))$ and we know the latter cannot capture the
eigenmodes of $\lapsurf$ correctly because of the issue of the
null-eigenspace discussed in Section~\ref{sec:nulleigenspace}.  Thus, it is not surprising to find that
the discrete operator $\DEunst$ also
does a poor job of approximating the Laplace--Beltrami operator.  For
example, Figure~\ref{fig:spectrum_hist_stab_unstab} shows that
$\DEunst$ has many eigenvalues close to zero, including some with
imaginary parts and negative real parts.
Notably, the matrix  $\DEunst$ is also ill-suited for time-dependent
calculations; for example, \cite{cbm:icpm,cbm:phd} show that
implicit methods built on the matrix $\DEunst$ have very strict
stability time-stepping restrictions and are not competitive with explicit schemes.

\subsection{A modified discretization}  \label{sec:stabCPM}

We will approximate the operator $\lapsharp$ using a matrix $\DEstab$
as described next.  This approach will yield a
convergent algorithm for the eigenfunctions and eigenvalues of
the Laplace--Beltrami operator $\lapsurf$.

Approximating $\lapsharp$ is straightforward and follows from the
definition \eqref{lapSharpDefn}.  When approximating $\lapsurf$ at the
point $\x_i$ with a finite difference scheme, we map only the
neighboring points $\x_j$ of the stencil back to their closest points
$\cp(\x_j)$ and use $u_i$ itself instead of $u(\cp(\x_i))$
\cite{cbm:icpm}.
This special treatment of the diagonal elements (which corresponds to
$\myeps=\Delta x$) yields a new $m \times m$ matrix \cite{cbm:icpm}
\begin{align*}
  \DEstab
  := \diag{\mDlap} + (\mDlap - \diag{\mDlap}) \m{E}
  = \raisebox{-1.05cm}{\includegraphics{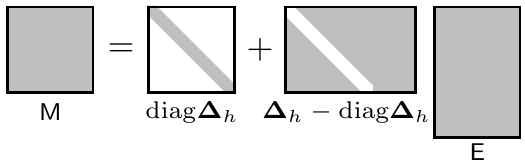}}\,.
\end{align*}
Figure~\ref{fig:spectrum_hist_stab_unstab} gives the results of an
experiment which contrasts the behavior of $\DEstab$ and $\DEunst$.
We find that the spectrum of $\DEstab$ matches that of $\lapsurf$ on a
semi-circular domain in $\Real^2$, whereas the spectrum of $\DEunst$
does not.

\begin{figure}
  \centering{%
    \includegraphics{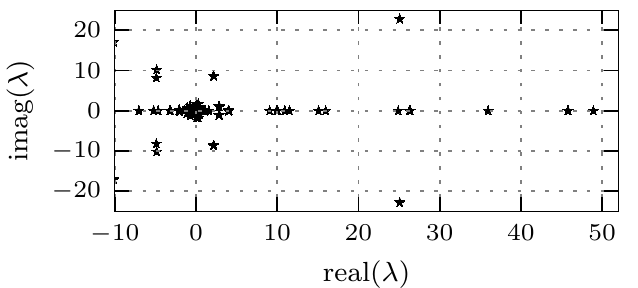}%
    \includegraphics{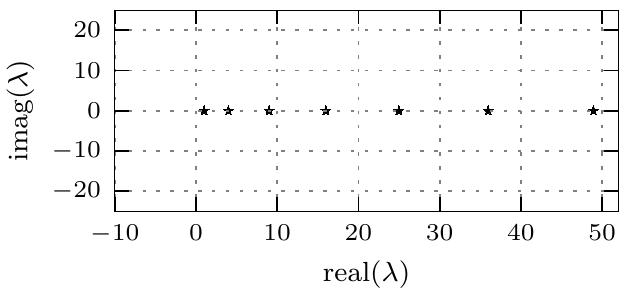}\\
    \includegraphics{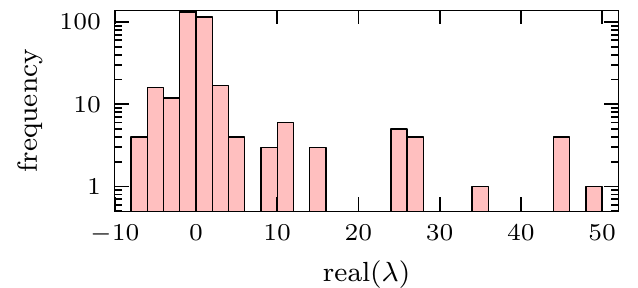}%
    \includegraphics{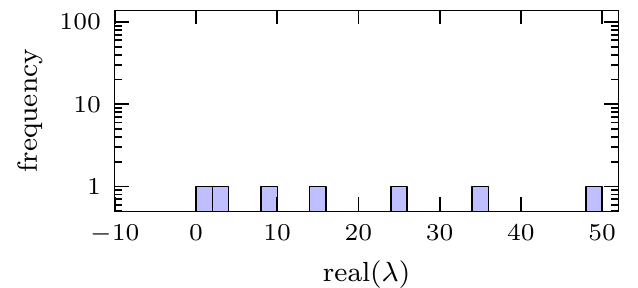}%
  }
  \caption{Eigenvalues and histograms of their distribution for the
    matrices $\DEunst$ (left) and $\DEstab$ (right).  The geometry is
    a unit semi-circle in 2D with homogeneous Dirichlet boundary
    conditions, $\m{E}$ with bicubic interpolation ($p=3$), and a mesh
    spacing of $\dx = \frac{1}{32}$.  Note the large number of
    eigenvalues near zero for $\DEunst$, whereas $\DEstab$ correctly
    captures the spectrum of $1,4,9,16,25,\ldots$.}
  \label{fig:spectrum_hist_stab_unstab}
\end{figure}

\paragraph{Remark}
The computational band $\uband$ must be chosen so that it contains the
interpolation stencil for $\cp(\x)$ for all $\x$ needed in the
calculation.  Figure~\ref{fig:egg} shows the computational grid for an
egg-shaped curve $\surf$ illustrating the list of grid points
$\uband$.  An algorithm for the construction of the list $\uband$ is
given in \cite{cbm:icpm}.

\begin{figure}
  \centering{%
    \includegraphics[width=0.36\textwidth]{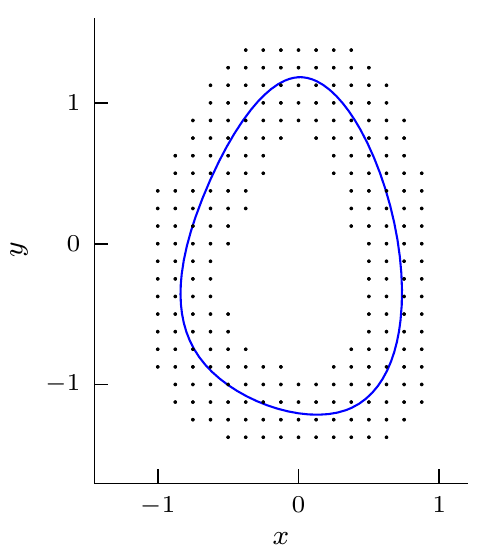}%
    \includegraphics[width=0.36\textwidth]{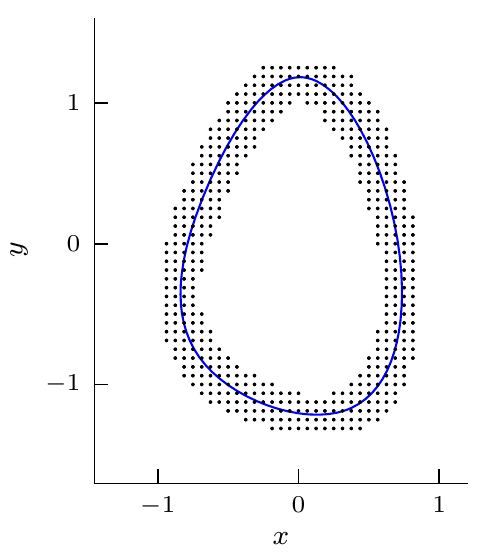}%
  }
  \caption{Computational grids for the Closest Point Method for an
    egg-shaped surface $\surf$ with $\dx=0.125$ (left) and
    $\dx=0.0625$ (right).  Second-order finite differences and degree
    $p=3$ interpolation.}
  \label{fig:egg}
\end{figure}

\subsection{Computing eigenvalues and eigenmodes}

Given the discrete operator $\DEstab$ which approximates the
Laplace--Beltrami operator on a surface $\surf$, we compute a spectral decomposition
\begin{align}
  \DEstab = \m{Q} \Lambda \m{Q}^{-1},
\end{align}
where $\diag{\Lambda}$ are the eigenvalues of $\DEstab$ and the
columns of $\m{Q}$ are the eigenvectors.  These eigenvalues and
eigenvectors are the respective approximations of the eigenvalues and
eigenfunctions of the Laplace--Beltrami operator on $\surf$.

\subsubsection{Implementation}

The final step of our algorithm---computing the spectral decomposition
of a matrix $\DEstab$---is a well-known problem in numerical linear
algebra (e.g., \cite{GolubvanLoan:matrix, Trefethen:numlinalg:1997}).
For example, in \Matlab we can use the function \code{eig()} to
compute the complete decomposition or the function \code{eigs()} to
compute only part of the spectrum.  \Matlab's \code{eig()} calls
LAPACK routines \cite{LAPACK} and \code{eigs()} calls ARPACK
\cite{ARPACKUserGuide} which makes use of the sparsity of the matrix
$\DEstab$ (in fact, it only requires a function which returns a
matrix-vector product, although in this work we explicitly form the
matrix $\DEstab$).  Many of our calculations are performed in Python
using SciPy \cite{scipy} and NumPy \cite{Oliphant:2007:numpy} where we
also make use of ARPACK via \code{scipy.sparse.linalg.eigen.arpack}.

The eigenfunctions are returned as vectors over the list of points
$\uband$.  In practice, the interpolation scheme from the closest
point extension can be reused to evaluate the eigenfunctions at any
desired points on the surface for plotting or other purposes.

\subsubsection{Degree of interpolation}
For the Closest Point Method to achieve the full order of accuracy of
the underlying finite difference scheme (say $q$), the degree of
interpolation $p$ must be chosen large enough so that the interpolant
itself can be differentiated sufficiently accurately.  For the
second-order Laplace--Beltrami problems in this work, we need $p \ge q
+ 1$ \cite{Ruuth/Merriman:jcp08:CPM,cbm:icpm}.

\section{Boundary conditions}
\label{sec:bc}

When applied to an open surface, the Closest Point Method propagates
boundary values into the embedding space along directions normal to
the boundary, yielding homogeneous Neumann boundary conditions
\cite{Ruuth/Merriman:jcp08:CPM}.  An analogous method for Dirichlet
boundary conditions is similarly straightforward: instead of
propagating out the interpolated values at boundary points the
prescribed boundary conditions are propagated out
\cite{Ruuth/Merriman:jcp08:CPM}.  While these methods are simple, they
are only first-order accurate, which is lower-order than the
discretizations of the Laplace--Beltrami operator discussed in this
paper.  Fortunately, a simple modification of the closest point
function can be introduced to obtain a second-order accurate
discretization for boundary conditions of Neumann or Dirichlet type.

Assume a smooth surface, and consider the following modification of
the closest point function,
\begin{equation}   \label{eq:cpbar}
  \bar{\cp}(\x) := \cp\big(\x + 2(\cp(\x) - \x)\big)
   = \cp\big(2\cp(\x) - \x\big).
\end{equation}
As is illustrated in Figure~\ref{fig:cpbar}, whenever $\cp(\x)$ is a
point in the interior of the surface, the line between $2\cp(\x)-\x$
and $\cp(\x)$ is orthogonal to the surface.  This implies that
$\bar{\cp}(\x) = \cp(2\cp(\x)-\x) = \cp(\x)$, at least in a neighborhood
of the surface.
Conversely, if $\bar{\cp}(\x_g) \neq
\cp(\x_g)$ for a point $\x_g$ then $\bar{\cp}(\x_g)$ is an interior
point of the surface\footnote{At least for $\x_g$ in a neighborhood
  of a sufficiently well-behaved surface: for example, for $\x_g$ far
  from one boundary of a curve segment, $\cp(2\cp(\x_g) - \x_g)$ might
  be another boundary point instead of an interior point.}  and $\cp(\x_g)$ is
a boundary point.  For a straight line or a planar surface,
$u(\bar{\cp}(\x_g))$ gives the mirror value for $\x_g$, while for a
general, curved surface it gives an approximate mirror value.  In
correspondence to the terminology for codimension-zero regions with
boundaries, we call a point $\x_g$ a \emph{ghost point} if $\cp(\x_g)
\neq \bar{\cp}(\x_g)$ (and note this terminology differs slightly from
\cite{cbm:icpm}).

\begin{SCfigure}
  \centering{
    \includegraphics[width=0.54\textwidth]{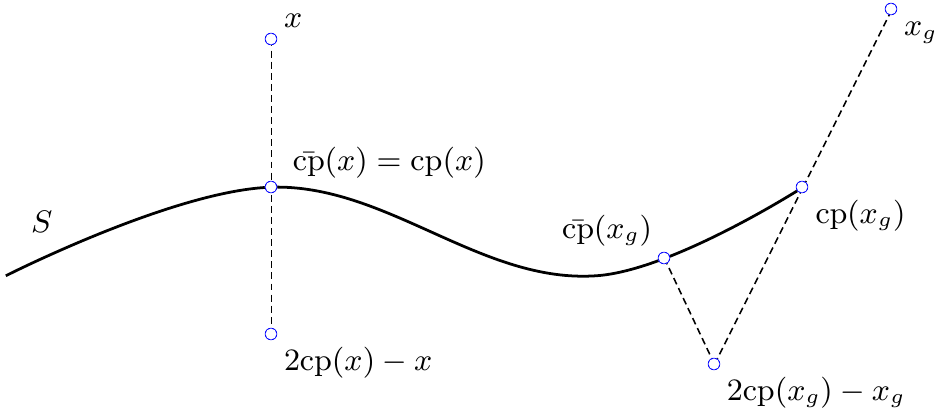}%
  }
  \caption{The function $\bar{\cp}$.  For the point $\x$, the closest
    point $\cp(\x)$ lies in the interior of the surface and for such a
    point, $\bar{\cp}(\x) = \cp(\x)$.  However, for $\x_g$, the
    closest point $\cp(\x_g)$ lies on the boundary of the surface and
    $\bar{\cp}(\x_g)$ lies in the interior, with $\bar{\cp}(\x_g)$ and
    $\x_g$ roughly equidistant from $\cp(\x_g)$.}
  \label{fig:cpbar}
\end{SCfigure}

Thus, replacing $\cp(\x)$ by $\bar{\cp}(\x)$ in the Closest Point Method
does not change the treatment of interior points.  At ghost points, however,
(approximate) mirror values are used.  This yields a second-order approximation of homogeneous Neumann conditions;
no other modification of the method is required.  Homogeneous
Dirichlet boundary conditions are obtained by extending the function at ghost points by $-u(\bar{\cp}(\x_g))$.  By
analogy to second-order boundary conditions for codimension-zero
regions, we observe that an approximation of non-homogeneous Dirichlet
conditions is obtained by extending at ghost points by $u(\x_g) = 2 u(\cp(\x_g))
-u(\bar{\cp}(\x_g))$, where $u(\cp(\x_g))$ is a
prescribed value on the boundary.

We emphasize that the previous discussion relates to boundaries of the
\emph{surface} and the treatment of boundary conditions imposed
thereat, not to the boundary of the narrow band of points (e.g., in
Figure~\ref{fig:egg}) where no artificial boundary condition need be
applied \cite{Ruuth/Merriman:jcp08:CPM, cbm:icpm}.

\section{Numerical Results}
\label{sec:numerical_results}

We present a series of numerical examples to demonstrate the
effectiveness of our approach.

\subsection{Numerical convergence studies}

We consider the egg-shaped curve in Figure~\ref{fig:egg} with
arclength $2\pi$ as a test case.  The Laplace--Beltrami eigenfunctions
and eigenvalues in this case are the same as those of $u''=\lambda u$,
where $u$ is $2\pi$ periodic, namely $u=\cos\left(n s + \beta\right)$
and $\lambda = n^2$ for $n \in \mathbb{Z} \ge 0$ where $s$ represents
arclength and $\beta$ is a phase shift.  The closest point function
for this curve was determined using a numerical optimization procedure
based on Newton's method.

Figure~\ref{fig:egg_ew_conv_study} shows convergence studies in $\dx$
for the first few eigenvalues on the egg-shaped domain.  For
second-order finite differences and degree-three interpolation we
observe second-order convergence.  That is, the Closest Point Method
approximates the eigenvalues with error $\bigoh{\dx^2}$.  Note that
the error increases (and this is true even if one measures relative
error) for the larger eigenvalues, but still shows second-order
convergence.  Figure~\ref{fig:egg_ew_conv_study} also shows that using
fourth-order finite differences and degree-five interpolation, we get
fourth-order accurate approximation to the eigenvalues.

\begin{figure}
  \centering{
    \subfloat[2nd-order FD, $p=3$ interp.]{\includegraphics[width=0.45\textwidth]{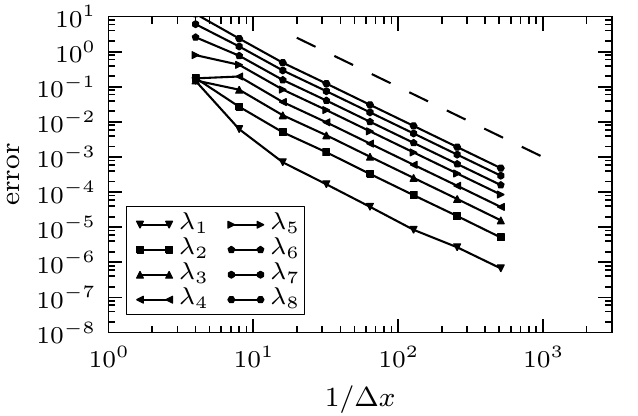}}
    \subfloat[4th-order FD, $p=5$ interp.]{\includegraphics[width=0.45\textwidth]{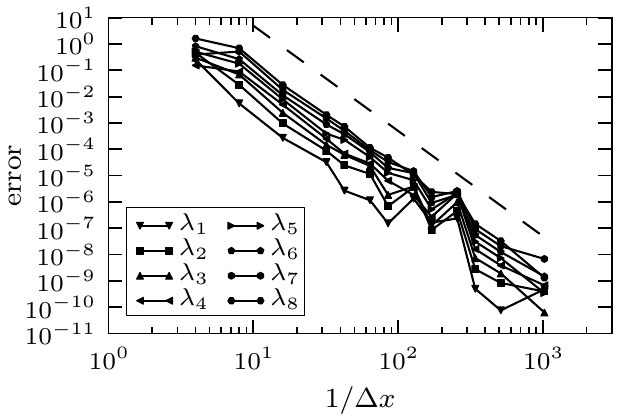}}
  }
  \caption{Numerical convergence studies for eight of the
    Laplace--Beltrami eigenvalues ($\lambda_n=n^2$) of a closed
    egg-shaped curve in 2D.  The dashed reference lines have slope two
    (left) and four (right).  Note second-order accuracy using
    second-order finite differences with degree $p=3$ interpolation
    (left) and fourth-order accuracy using fourth-order finite
    differences with degree $p=5$ interpolation (right).  The lack of
    smoothness in the fourth-order results plot may be because the
    underlying spline curve itself is only $C^2$ smooth; results on a
    circle (not included) are smoother.}
  \label{fig:egg_ew_conv_study}
\end{figure}

In further numerical tests (not included), we found that degree $p=2$
interpolation with second-order finite differences, and degree $p=4$
interpolation with fourth-order finite differences, give approximately
second- and fourth-order convergence, respectively.  These are better
than expected by one order of accuracy, as originally observed in
\cite{cbm:icpm}.

\subsubsection{Boundary Conditions}

In this section we verify using convergence studies that our treatment
of boundary conditions introduced in Section~\ref{sec:bc} achieves the
expected order of accuracy.  For our tests, we use the curve shown in
Figure~\ref{fig:curve_w_ends} parameterized as $(x(t),y(t)) = (t,\cos
t)$ for $t \in [1/4, 4]$.  We apply homogeneous Neumann and Dirichlet
boundary conditions to the ends.  Again, the exact eigenvalues and
eigenfunctions can be determined analytically by considering the
corresponding problem on an interval.  For the former, the arclength
of the curve is required and can be determined in terms of elliptic
integrals.

\begin{figure}
  \centering{
    \subfloat[2nd-order FD, $p=3$ interp.]{%
      \includegraphics[width=0.48\textwidth]{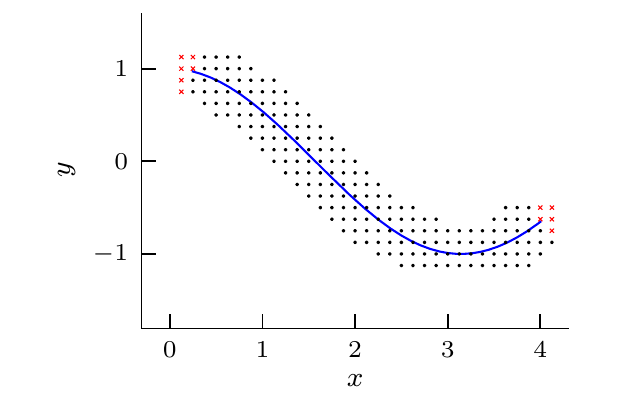}}
    \subfloat[4th-order FD, $p=5$ interp.]{%
      \includegraphics[width=0.48\textwidth]{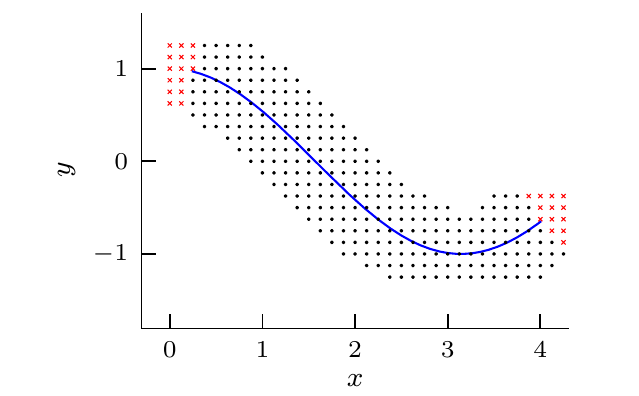}}
  }
  \caption{A curve and the corresponding computational grids for the
    Closest Point Method with second-order (left) and fourth-order
    (right) finite difference schemes with $\dx=0.125$.  Grid points
    marked with small crosses are the ghost points involved in
    implementing boundary conditions.}
  \label{fig:curve_w_ends}
\end{figure}

The original Closest Point Method, which does not explicitly treat
boundaries, gives a first-order approximation to homogeneous boundary
conditions \cite{Ruuth/Merriman:jcp08:CPM}.
See, for example, Figure~\ref{fig:curve_ev_neumann_conv_study} where
it is observed that this trivial treatment of the boundaries limits
the overall accuracy in the eigenvalues to first-order.
Figure~\ref{fig:curve_ev_neumann_conv_study} also gives results using
the new ``$\bar{\cp}$'' approach for Neumann boundary conditions
described in Section~\ref{sec:bc}; we find that this method attains
the second-order accuracy of the underlying Cartesian finite
difference scheme.  Note that the $\bar{\cp}$ approach is also
effective for Dirichlet boundary conditions.  For example, in
Figure~\ref{fig:curve_ev_dirichlet_conv_study} we find second-order
results for homogeneous Dirichlet boundary conditions using
$-u(\bar{\cp}(\x_g)$ for ghost points (as in Section~\ref{sec:bc}).

The $\bar{\cp}$ approach is designed to give second-order
approximations to boundary conditions.  Thus it is expected, and
observed in Figure~\ref{fig:curve_ev_dirichlet_conv_study}, that even
with higher-order finite difference schemes and high-degree
interpolation, the results will generally be second-order accurate in
the presence of boundary conditions.  Third- and higher-order
approximation of boundary conditions may also be contemplated; while
we do not investigate such methods here we note that approximations of
this type will require a replacement for $\bar{\cp}$ that incorporates
the curvature of the surface near the boundary.

\begin{figure}
  \centering{%
    \includegraphics[width=0.45\textwidth]{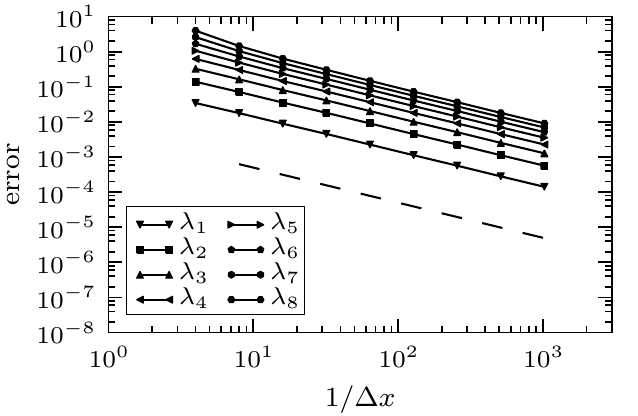}%
    \quad%
    \includegraphics[width=0.45\textwidth]{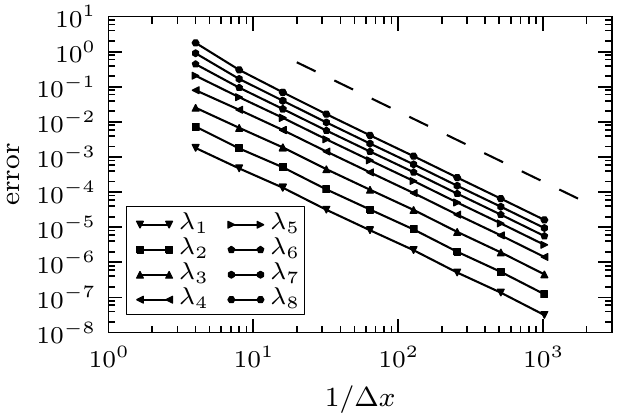}%
  }
  \caption{Numerical convergence studies for the first eight non-zero
    Laplace--Beltrami eigenvalues of a curve in 2D with Neumann
    boundary conditions applied at both ends.  The left figure uses
    the original closest point function to impose a homogeneous
    Neumann boundary condition, and we note the results are only
    first-order accuracy (dashed reference line has slope one).  The
    right figure uses the modified ``$\bar{\cp}$'' approach and
    exhibits second-order accuracy (dashed reference line has slope
    two).  Both computations use second-order finite differences and
    degree $p=3$ interpolation.}
  \label{fig:curve_ev_neumann_conv_study}
\end{figure}

\begin{figure}
  \centering{%
    \includegraphics[width=0.45\textwidth]{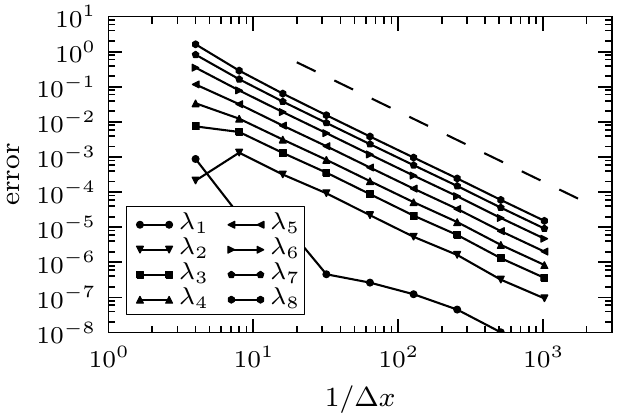}%
    \quad%
    \includegraphics[width=0.45\textwidth]{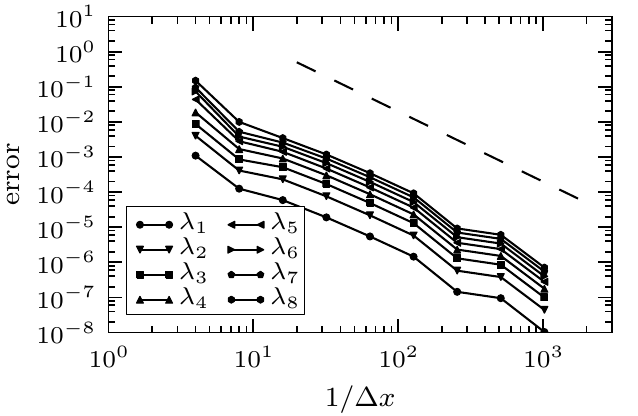}%
  }
  \caption{Numerical convergence study for the first few eigenvalues
    on a curve with Dirichlet boundary conditions applied at both
    ends.  The boundaries are dealt with using the ``$\bar{\cp}$'' and
    $-u$ discretization with degree $p$ interpolation.  The left
    figure uses second-order finite differences (and degree $p=3$
    interpolation) and achieves second-order accuracy (the dashed
    reference lines have slope two).  The right figure uses
    fourth-order finite differences (and degree $p=5$ interpolation)
    but the overall accuracy is limited by the second-order treatment
    of the boundary conditions.}
  \label{fig:curve_ev_dirichlet_conv_study}
\end{figure}

\subsubsection{Conditioning}
\label{sec:conditioning}

Table~\ref{tab:cond} shows that the condition number of the matrices
used in our computations scales like $\bigoh{\frac{1}{\dx^2}}$.  Thus
the conditioning is the same as for standard Cartesian finite
difference schemes.  The ``$\bar{\cp}$'' treatment of boundary
conditions does not have a significant effect on conditioning.

\begin{SCtable}
  \caption{Condition numbers (2-norm) for the matrix $\DEstab$.
    Tested for the Laplace--Beltrami operator on a open curve embedded
    in 2D with Dirichlet boundary conditions at both ends, using
    second-order finite differences, degree $p=3$ interpolation, and
    the ``$\bar{\cp}$'' treatment of the boundary conditions.}
  \label{tab:cond}
  \centering
  \begin{tabular}{ccl}
    \hline
    $\dx$ & size of $\DEstab$ & $\kappa(\DEstab)$ \\
    \hline
    0.25      & 76   $\times$ 76   & 289    \\
    0.125     & 140  $\times$ 140  & 1154   \\
    0.0625    & 268  $\times$ 268  & 4608   \\
    0.03125   & 524  $\times$ 524  & 19304  \\
    0.015625  & 1036 $\times$ 1036 & 75543  \\
    0.0078125 & 2060 $\times$ 2060 & 326633 \\
    \hline
  \end{tabular}
\end{SCtable}

\subsubsection{Behavior for high-frequency modes}
\label{sec:high_freq_ews}

The large number of spurious eigenvalues near $\frac{2d}{\Delta x^2}$
in Figure~\ref{fig:high_freq_ews} reflect the singular behavior of
Equation~\eqref{embeddedEigenfunc} at $\lambda = \frac{2d}{\Delta
  x^2}$ (c.f., Figure~\ref{fig:spectrum_hist_stab_unstab} where the
problem instead occurs near the origin).  Notably, it is these
spurious eigenvalues which possess nonzero imaginary components,
reflecting the fact that the eigenvalue problem \eqref{eq:LBevstab} is
not self-adjoint.  Because these spurious eigenvalues correspond to
highly oscillatory eigenfunctions which are quite close to the Nyquist
frequency (i.e., with eigenvalues near $\frac{\pi^2}{\Delta x^2}$), it
is not surprising that numerical difficulties arise for such modes.
Nonetheless, any particular higher-frequency modes may be resolved by
appropriately refining~$\Delta x$.

\begin{figure}
  \centerline{%
    \includegraphics[width=0.32\textwidth]{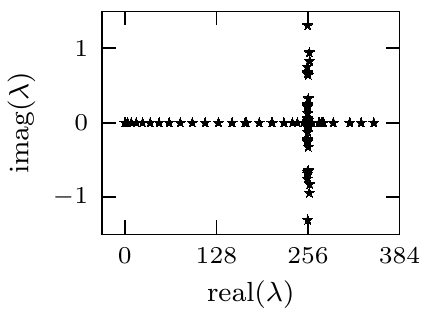}%
    \includegraphics[width=0.32\textwidth]{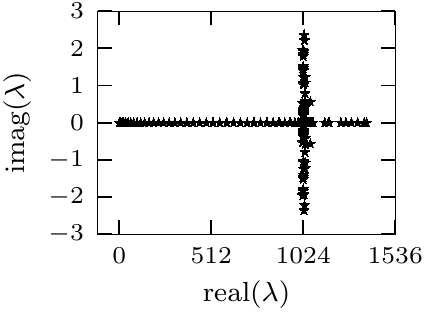}%
    \includegraphics[width=0.32\textwidth]{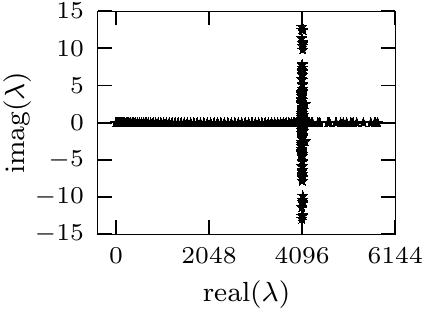}%
  }
  \caption{Computed spectra of the Laplace--Beltrami operator on a
    unit circle using $\dx=\frac{1}{8}, \frac{1}{16}$, and
    $\frac{1}{32}$ (from left-to-right).  The method uses second-order
    finite differences and degree $p=3$ interpolation.  Note a large
    number of extraneous complex eigenvalues near $\frac{4}{\Delta
      x^2}$ (i.e., 256, 1024, and 4096): these correspond to the
    singularity in \eqref{embeddedEigenfunc} and can be controlled by
    further resolving $\Delta x$.}
  \label{fig:high_freq_ews}
\end{figure}

\subsection{Eigenfunction computations}

The following computations were  performed in Matlab and SciPy \cite{scipy}.
Visualizations were carried out with Matlab, VTK \cite{vtk1998} and MayaVi
\cite{mayavi2}.

\subsubsection{Hemispherical harmonics}
Figure~\ref{fig:hemisphere} shows several Laplace--Beltrami eigenmodes
of a unit hemisphere with homogeneous Neumann boundary conditions on
the equator.

\begin{figure}
  \parbox{0.48\textwidth}{%
    \centerline{%
      \includegraphics[width=0.12\textwidth]{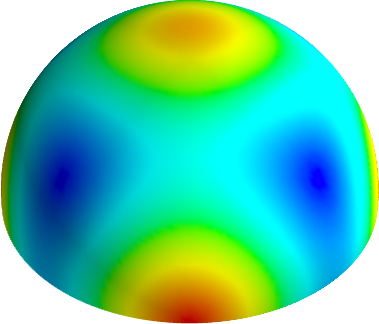}
      \includegraphics[width=0.12\textwidth]{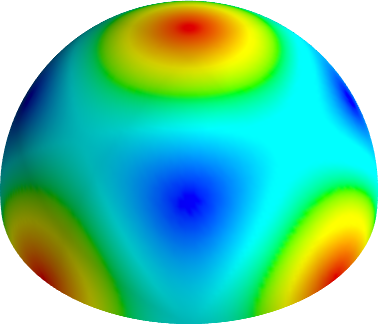}
      \includegraphics[width=0.12\textwidth]{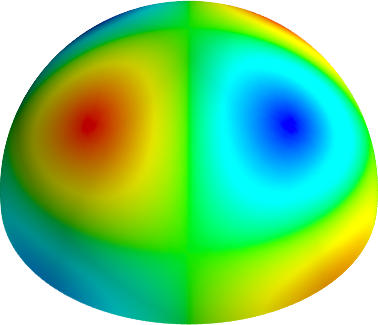}}
    \centerline{%
      \includegraphics[width=0.12\textwidth]{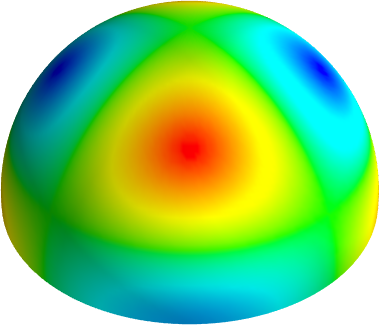}
      \includegraphics[width=0.12\textwidth]{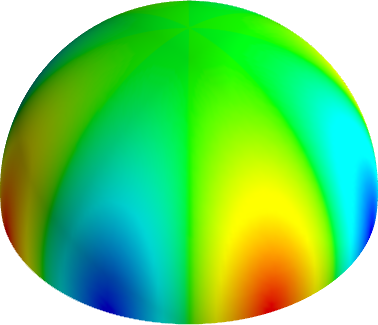}}}
  \parbox{0.48\textwidth}{%
    \includegraphics[width=0.45\textwidth]{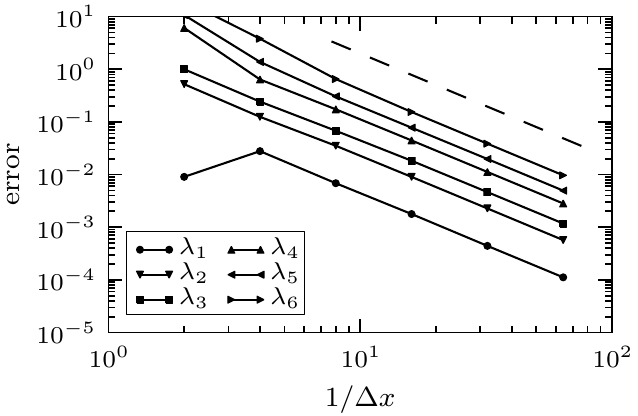}}%
  \caption{The left image shows several hemispherical harmonics with
    eigenvalue 20 computed
    with second-order finite differences, degree $p=3$ interpolation,
    and $\Delta x = \frac{1}{64}$.  The right image is a convergence study
    for the first six eigenvalues ($\lambda_n = n(n+1)$), also using
    second-order finite differences and degree $p=3$ interpolation
    (the dashed reference line has slope two).}
  \label{fig:hemisphere}
\end{figure}

\subsubsection{Eigenkaninchen}
Figure~\ref{fig:bunny} shows Laplace--Beltrami eigenmodes of the
surface of the Stanford Bunny \cite{StanfordBunny}.  The closest point
function for this and other triangulated surfaces can be computed in
an efficient, straightforward manner \cite{cbm:lscpm}.
\begin{figure}
  \centerline{%
    \rule[-0.5ex]{0pt}{1ex}%
    \includegraphics[width=0.195\textwidth]{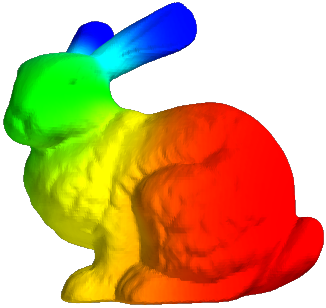}%
    \includegraphics[width=0.195\textwidth]{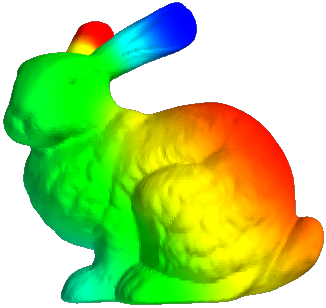}%
    \includegraphics[width=0.195\textwidth]{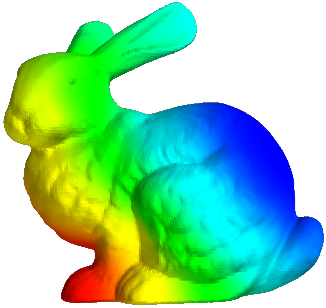}%
    \includegraphics[width=0.195\textwidth]{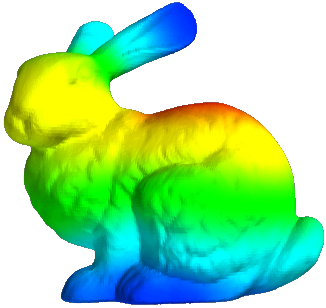}%
    \includegraphics[width=0.195\textwidth]{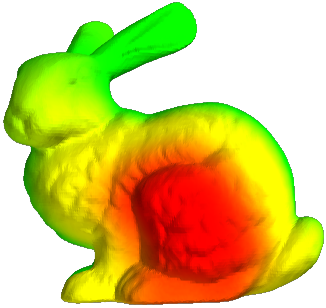}}%
  \centerline{%
    \includegraphics[width=0.195\textwidth]{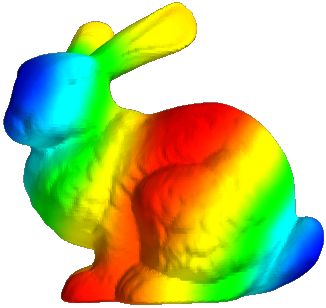}%
    \includegraphics[width=0.195\textwidth]{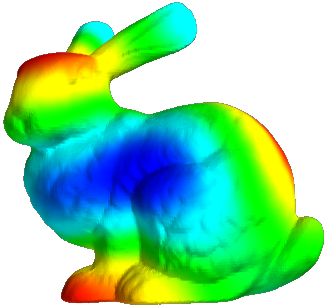}%
    \includegraphics[width=0.195\textwidth]{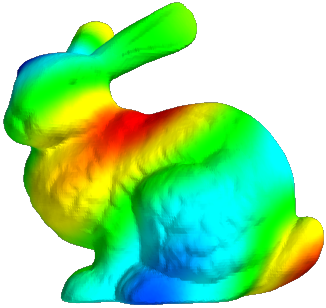}%
    \includegraphics[width=0.195\textwidth]{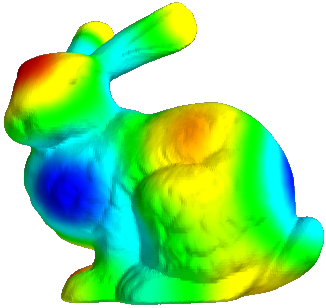}%
    \includegraphics[width=0.195\textwidth]{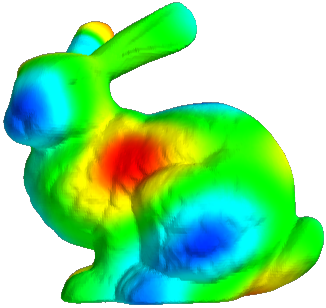}}%
  \caption{A selection of eigenkaninchen: Laplace--Beltrami eigenmodes
    on the surface of the Stanford Bunny \cite{StanfordBunny}.
    The results are computed via the Closest Point Method using second-order finite
    differences, degree $p=3$ interpolation, and $\Delta x=0.1$, where
    the bunny is roughly two units long. }
  \label{fig:bunny}
\end{figure}

\subsubsection{M\"obius strip}
Figure~\ref{fig:mobius} shows some Laplace--Beltrami eigenmodes of a
M\"obius strip.  Dirichlet boundary conditions are applied at the
boundary of the M\"obius strip.  The surface is embedded in $\Real^3$
and the closest point function for each grid point is computed from a
parameterization using a numerical optimization procedure minimizing
the square of the distance function.  Note that the non-orientable
nature of the M\"obius strip poses no difficultly for the Closest
Point Method.

\begin{SCfigure}
  \newcommand{\Lshfigsize}{1in}%
  \parbox{3.2in}{%
    \centerline{%
      \rule[-2.5pt]{0pt}{1ex}%
      \includegraphics[width=\Lshfigsize]{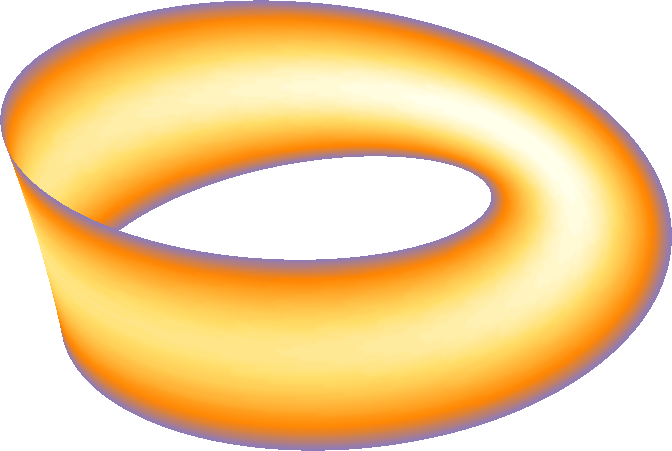}
      \includegraphics[width=\Lshfigsize]{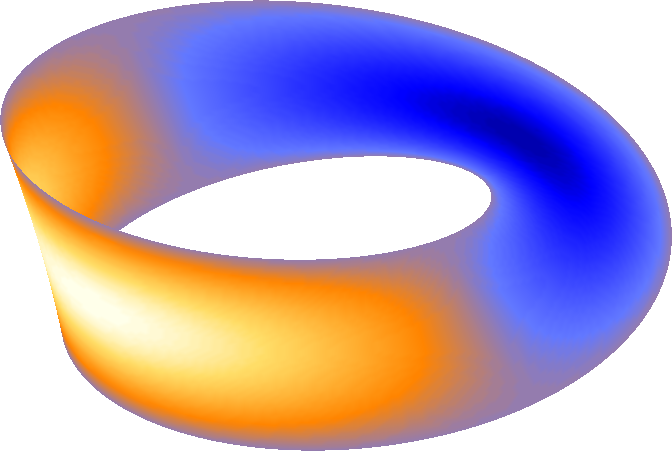}
      \includegraphics[width=\Lshfigsize]{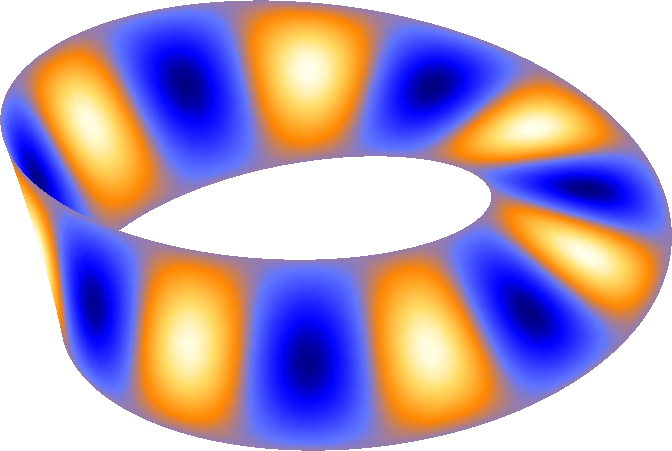}}%
    \centerline{%
      \rule[-2.5pt]{0pt}{1ex}%
      \includegraphics[width=\Lshfigsize]{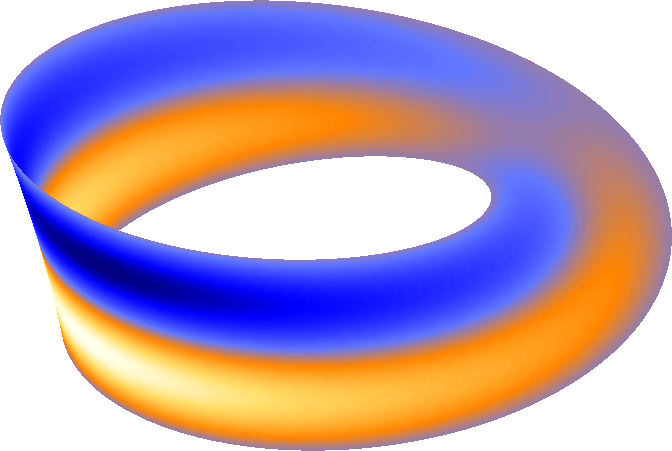}
      \includegraphics[width=\Lshfigsize]{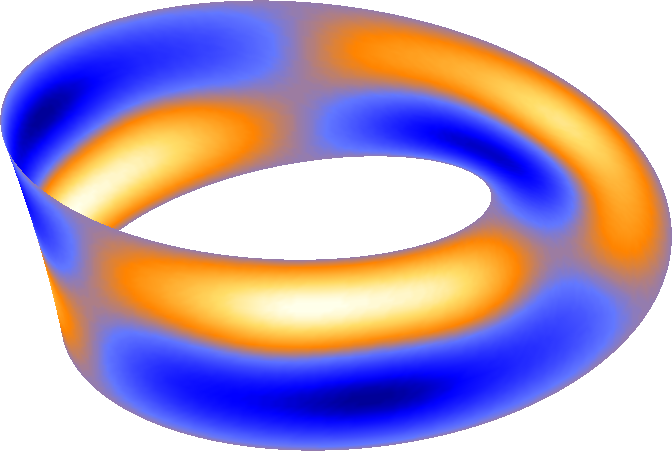}
      \includegraphics[width=\Lshfigsize]{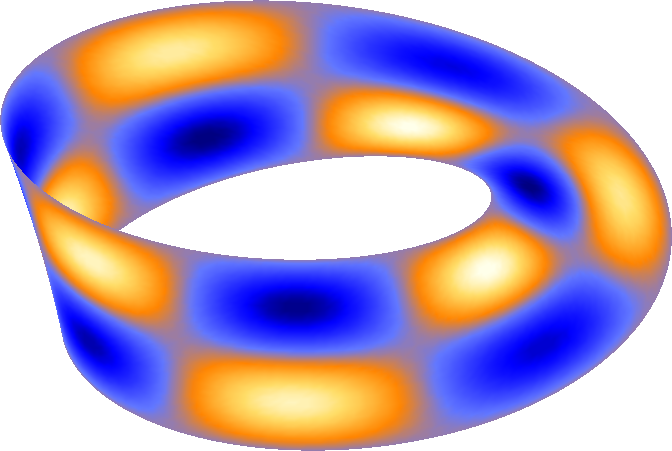}}%
    \centerline{%
      \rule[-2.5pt]{0pt}{1ex}%
      \includegraphics[width=\Lshfigsize]{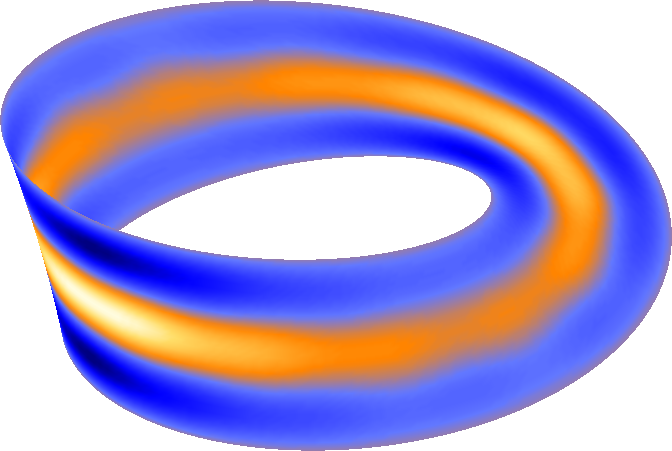}
      \includegraphics[width=\Lshfigsize]{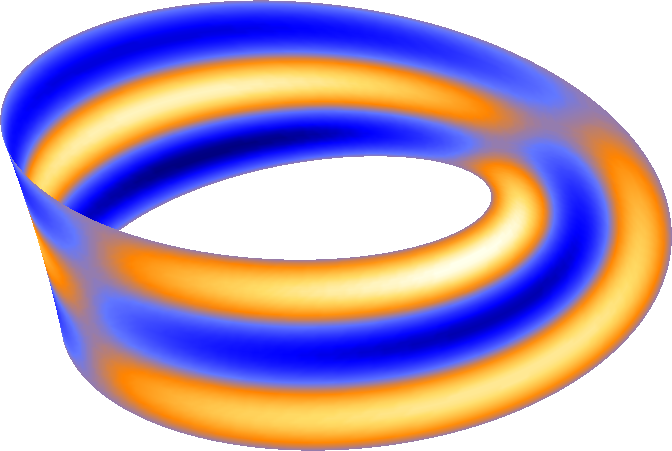}
      \includegraphics[width=\Lshfigsize]{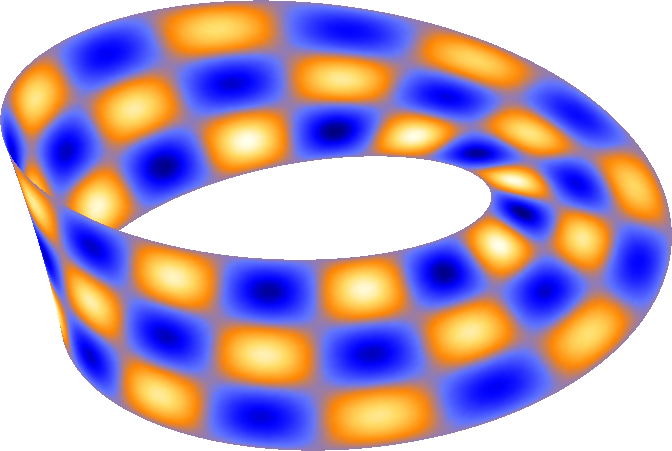}}%
  }
\caption{A selection of Laplace--Beltrami eigenmodes of a M\"obius
  strip computed with the Closest Point Method.  Calculations use
  second-order finite differences, $p=3$, $\Delta x=0.1$, and the
  M\"obius strip is about 2 units in ``diameter''.}
  \label{fig:mobius}
\end{SCfigure}

\subsubsection{L-shaped Domain}

The Closest Point Method works on surfaces of various codimension
\cite{Ruuth/Merriman:jcp08:CPM, cbm:icpm} and indeed solid shapes in
$\Real^2$ or $\Real^3$ are surfaces of codimension-zero.
Figure~\ref{fig:Lshape} shows an eigenmode computation on an L-shaped
domain, where zero Dirichlet boundary conditions are imposed using the
``$\bar{\cp}$'' technique described in Section~\ref{sec:bc}.  In the interior of a
solid, $\cp(\x)=\x$ and so no interpolations are needed.
Furthermore, for a grid-aligned L-shaped domain, for any $\x_i$ outside the domain,
$\bar{\cp}(\x_i)$ turns out to be another grid point $\x_j$ (located
inside the domain as if the perimeter were a mirror) so no
interpolation step is necessary.  Thus in this special ``corner
case'', the Closest Point Method reduces to a standard finite
difference computation using ghost points to mirror the values along
the perimeter.  Interestingly, these reductions happen
automatically: no change in the code is needed.

\begin{SCfigure}
  \newcommand{\Lshfigsize}{0.52in}%
  \parbox{2.9in}{%
    \centerline{%
      \rule[-2.5pt]{0pt}{1ex}%
      \includegraphics[width=\Lshfigsize]{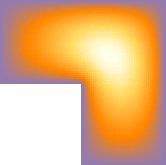}
      \includegraphics[width=\Lshfigsize]{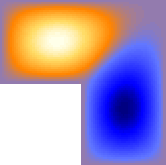}
      \includegraphics[width=\Lshfigsize]{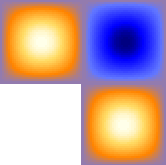}
      \includegraphics[width=\Lshfigsize]{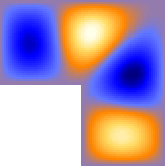}
      \includegraphics[width=\Lshfigsize]{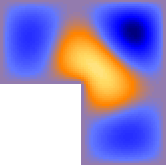}}%
    \centerline{%
      \rule[-2.5pt]{0pt}{1ex}%
      \includegraphics[width=\Lshfigsize]{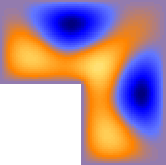}
      \includegraphics[width=\Lshfigsize]{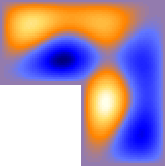}
      \includegraphics[width=\Lshfigsize]{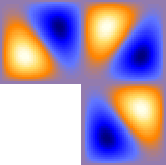}
      \includegraphics[width=\Lshfigsize]{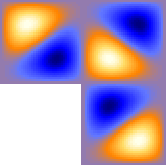}
      \includegraphics[width=\Lshfigsize]{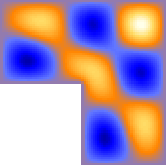}}%
    \centerline{%
      \rule[-2.5pt]{0pt}{1ex}%
      \includegraphics[width=\Lshfigsize]{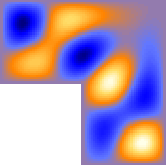}
      \includegraphics[width=\Lshfigsize]{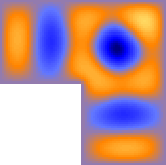}
      \includegraphics[width=\Lshfigsize]{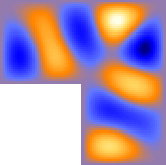}
      \includegraphics[width=\Lshfigsize]{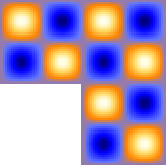}
      \includegraphics[width=\Lshfigsize]{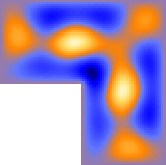}}%
  }
  \caption{The first 15 Laplacian eigenmodes of an L-shaped domain
    with Dirichlet boundary conditions.  The results are computed
    using the Closest Point Method with second-order finite
    differences using $\dx = 1/40$ (the domain is two units wide).}
  \label{fig:Lshape}
\end{SCfigure}

\subsubsection{Continental Africa}
Some Laplace--Beltrami eigenmodes of continental Africa are shown in
Figure~\ref{fig:africa}.  The results are computed directly on the
surface of the Earth (assumed to be a sphere).  Homogeneous Dirichlet
boundary conditions are applied to the coastline.  Finding the closest
point function involves projecting onto a bitmapped image of the Earth
\cite{NASA:visible_earth} where the continent was first manually
segmented.  It is interesting to note that these eigenmodes match very
closely the first ten eigenmodes of the L-shaped domain in
Figure~\ref{fig:Lshape}.

\begin{figure}
  \newcommand{\figsize}{0.19\textwidth}
  \centerline{%
    \rule[-2.5pt]{0pt}{1ex}%
    \includegraphics[width=\figsize]{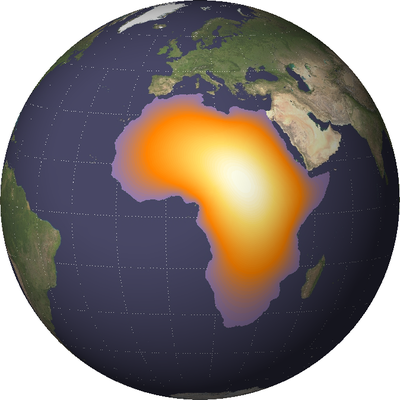}
    \includegraphics[width=\figsize]{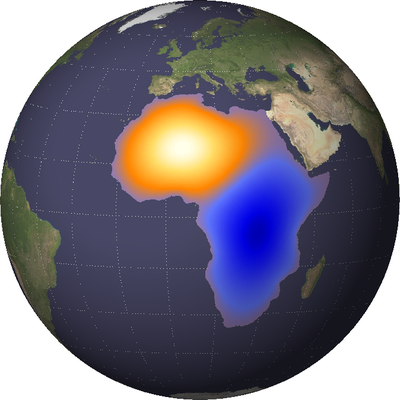}
    \includegraphics[width=\figsize]{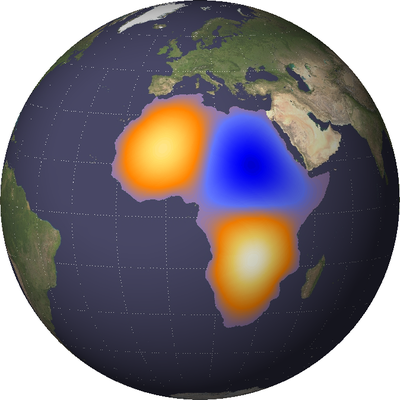}
    \includegraphics[width=\figsize]{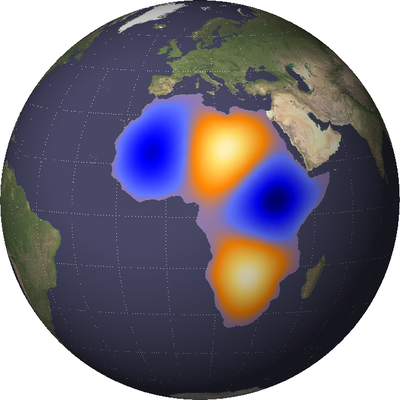}
    \includegraphics[width=\figsize]{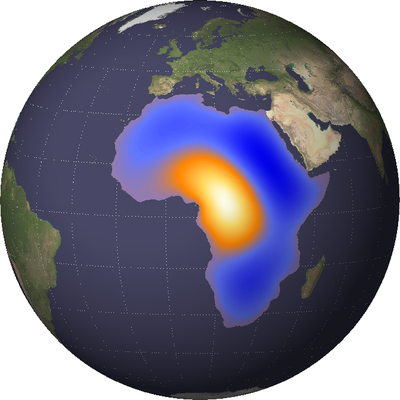}}%
  \centerline{%
    \rule[-2.5pt]{0pt}{1ex}%
    \includegraphics[width=\figsize]{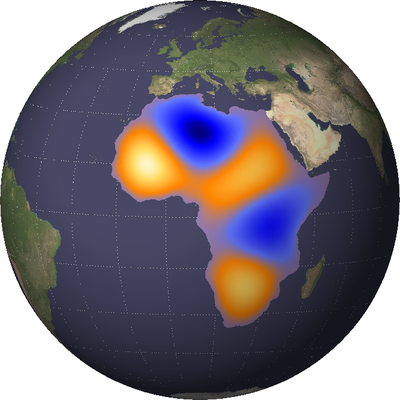}
    \includegraphics[width=\figsize]{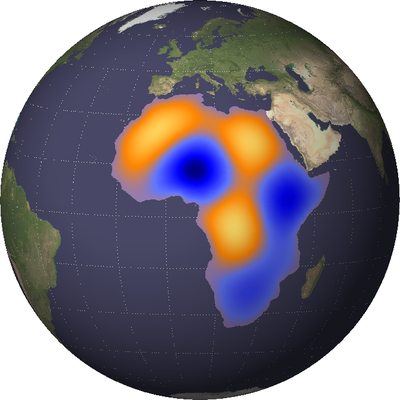}
    \includegraphics[width=\figsize]{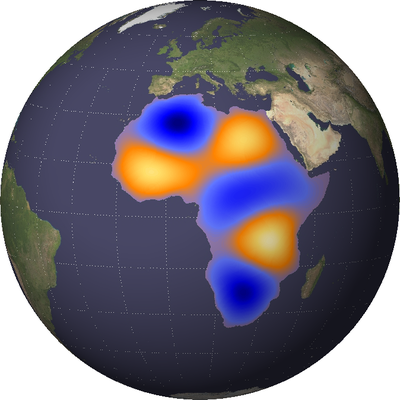}
    \includegraphics[width=\figsize]{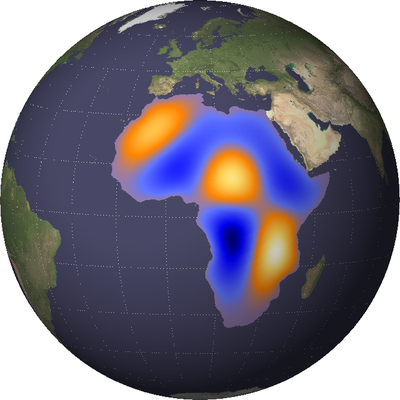}
    \includegraphics[width=\figsize]{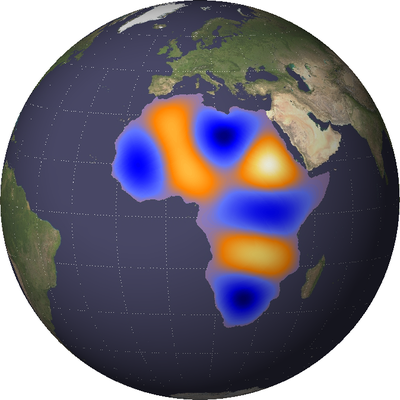}}%
  \caption{First ten eigenmodes of continental Africa.  The Earth has
    unit diameter and the computation uses second-order finite
    differences with $\Delta x=1/40$ and degree $p=3$ interpolation.}
  \label{fig:africa}
\end{figure}

\section{Summary and Conclusions}
\label{sec:conclusions}
Through a series of convergence studies and computational examples, we
have shown that the Closest Point Method is an effective method for
computing the spectra and eigenfunctions of the Laplace--Beltrami
operator on rather general surfaces.  The basis of our approach is the
embedded eigenvalue problem \eqref{eq:LBevstab} which, when
discretized using standard, centered finite difference methods and
Lagrange interpolation, yields a nonsymmetric matrix eigenvalue
problem which can be solved using standard software.  Fortunately,
this lack of symmetry is not a concern in many practical situations
since only the highest frequency modes have a significant imaginary
component.  We are currently investigating a finite element Closest
Point Method which would lead to symmetric matrices.

For eigenvalue problems on open surfaces with Dirichlet or homogeneous
Neumann boundary conditions, we have introduced second-order accurate
schemes.  These methods require only a simple change to the
closest point extension and are straightforward to implement.

Although we have focused on the Laplace--Beltrami operator in this
work, the Closest Point Method is applicable to many other surface
differential operators
\cite{Ruuth/Merriman:jcp08:CPM,cbm:lscpm,cbm:phd,cbm:icpm}.  This
suggests that the approach presented here may be applicable to a
larger class of eigenvalue problems.

Finally, the techniques developed here are quite general and are
applicable beyond surface eigenvalue problems.  For example, future
work will include solving surface Poisson problems of the form
$-\lapsurf u = f$.

\subsection*{Acknowledgments}
The authors thank Bin Dong (UCSD), who motivated this work by
asking if the Closest Point Method could be used for surface
eigenvalue calculations.

\appendix

\section{Additional theorems}
\label{app:unst_case_thms}
As mentioned in Section~\ref{sec:illposed_ev_prob}, every solution to the
ill-posed embedding problem \eqref{eq:evCP}, restricted to the
surface, is a solution to the surface eigenvalue problem.  Conversely,
except for the $\lambda = 0$ case, every surface eigenfunction
corresponds to a unique solution of the embedding problem.  These
results are established by the following theorems.
\begin{mytheorem}  \label{th:evCPimpliesLBev}
  Suppose $v(\x)$ and $\lambda$ are a solution to the embedding
  eigenvalue problem \eqref{eq:evCP} and $\surf$ is a smooth surface.  Then $u:\surf \to \Real$
  defined by $u(\x) = v(\x)$ for $\x \in \surf$ is an eigenfunction of
  \eqref{eq:LBev} with eigenvalue $\lambda$.
\end{mytheorem}
\begin{proof}
This is a direct consequence of Theorem~\ref{cp_principle1}.
\end{proof}

\begin{mytheorem} \label{th:LBevimplesevCP}
  Let $\surf$ be a smooth surface.
  For every non-constant solution $u$ and $\lambda$ of
  \eqref{eq:LBev}, there exists a unique (up to a multiplicative
  constant) eigenfunction $v$ of \eqref{eq:evCP} with eigenvalue
  $\lambda$ which agrees with $u$ on $\surf$.  This eigenfunction is
  given by
    $v(\x) = -\frac{1}{\lambda}\lap(u(\cp(\x))).$
\end{mytheorem}
\begin{proof}
  Using our hypothesis that $v$ agrees with $u$ on $\surf$, we may
  solve for $v$ in \eqref{eq:evCP} to obtain the result.
\end{proof}


\bibliographystyle{elsarticle-num}
\bibliography{cbm,ruuth}

\begin{thebibliography}{10}
\expandafter\ifx\csname url\endcsname\relax
  \def\url#1{\texttt{#1}}\fi
\expandafter\ifx\csname urlprefix\endcsname\relax\def\urlprefix{URL }\fi
\expandafter\ifx\csname href\endcsname\relax
  \def\href#1#2{#2} \def\path#1{#1}\fi

\bibitem{Reuter2005}
M.~Reuter, F.-E. Wolter, N.~Peinecke, Laplace-spectra as fingerprints for shape
  matching, in: Proceedings of the 2005 ACM symposium on Solid and physical
  modeling, SPM '05, 2005, pp. 101--106.
\newblock \href {http://dx.doi.org/10.1145/1060244.1060256}
  {\path{doi:10.1145/1060244.1060256}}.

\bibitem{Belkin}
M.~Belkin, P.~Niyogi, Laplacian eigenmaps for dimensionality reduction and data
  representation, Neural Computation 15~(6) (2003) 1373--1396.
\newblock \href {http://dx.doi.org/10.1162/089976603321780317}
  {\path{doi:10.1162/089976603321780317}}.

\bibitem{Coifman2006}
R.~R. Coifman, S.~Lafon, Diffusion maps, Applied and Computational Harmonic
  Analysis 21~(1) (2006) 5--30.
\newblock \href {http://dx.doi.org/10.1016/j.acha.2006.04.006}
  {\path{doi:10.1016/j.acha.2006.04.006}}.

\bibitem{Seo}
S.~Seo, M.~Chung, H.~Vorperian, Heat kernel smoothing using
  {L}aplace--{B}eltrami eigenfunctions, in: T.~Jiang, N.~Navab, J.~Pluim,
  M.~Viergever (Eds.), Medical Image Computing and Computer-Assisted
  Intervention - MICCAI 2010, Vol. 6363 of Lecture Notes in Computer Science,
  Springer Berlin / Heidelberg, 2010, pp. 505--512.
\newblock \href {http://dx.doi.org/10.1007/978-3-642-15711-0_63}
  {\path{doi:10.1007/978-3-642-15711-0_63}}.

\bibitem{Reuter2009}
M.~Reuter, S.~Biasotti, D.~Giorgi, G.~Patan{\`{e}}, M.~Spagnuolo, Discrete
  {L}aplace--{B}eltrami operators for shape analysis and segmentation,
  Computers \& Graphics 33~(3) (2009) 381--390, iEEE International Conference
  on Shape Modelling and Applications 2009.
\newblock \href {http://dx.doi.org/10.1016/j.cag.2009.03.005}
  {\path{doi:10.1016/j.cag.2009.03.005}}.

\bibitem{Glowinski}
R.~Glowinski, D.~C. Sorensen, Computing the eigenvalues of the
  {L}aplace--{B}eltrami operator on the surface of a torus: A numerical
  approach, in: E.~O\~{n}ate, R.~Glowinski, P.~Neittaanm\"{a}ki (Eds.), Partial
  Differential Equations, Vol.~16 of Computational Methods in Applied Sciences,
  Springer Netherlands, 2008, pp. 225--232.
\newblock \href {http://dx.doi.org/10.1007/978-1-4020-8758-5_12}
  {\path{doi:10.1007/978-1-4020-8758-5_12}}.

\bibitem{Reuter:2006:shapeDNA}
M.~Reuter, F.-E. Wolter, N.~Peinecke, {L}aplace--{B}eltrami spectra as
  {`Shape-DNA'} of surfaces and solids, Computer-Aided Design 38~(4) (2006)
  342--366.
\newblock \href {http://dx.doi.org/10.1016/j.cad.2005.10.011}
  {\path{doi:10.1016/j.cad.2005.10.011}}.

\bibitem{Brandman:JSC2008:eigenvalues}
J.~Brandman, A level-set method for computing the eigenvalues of elliptic
  operators defined on compact hypersurfaces, J. Sci. Comput. 37 (2008)
  282--315.
\newblock \href {http://dx.doi.org/10.1007/s10915-008-9210-z}
  {\path{doi:10.1007/s10915-008-9210-z}}.

\bibitem{Bertalmio:2001}
M.~Bertalm{\'{\i}}o, L.-T. Cheng, S.~Osher, G.~Sapiro, Variational problems and
  partial differential equations on implicit surfaces, J. Comput. Phys. 174~(2)
  (2001) 759--780.

\bibitem{XuZhao:2003:PDEsMovIntf}
J.~Xu, H.~Zhao, An {E}ulerian formulation for solving partial differential
  equations along a moving interface, J. Sci. Comput. 19~(1) (2003) 573--594.
\newblock \href {http://dx.doi.org/10.1023/A:1025336916176}
  {\path{doi:10.1023/A:1025336916176}}.

\bibitem{GreerBertozzi:2006:biharmonic}
J.~B. Greer, A.~L. Bertozzi, G.~Sapiro, Fourth order partial differential
  equations on general geometries, J. Comput. Phys. 216~(1) (2006) 216--246.

\bibitem{Greer:2006}
J.~B. Greer, An improvement of a recent {E}ulerian method for solving {PDE}s on
  general geometries, J. Sci. Comput. 29~(3) (2006) 321--352.

\bibitem{Nemitz:Submitted}
O.~Nemitz, M.~B. Nielsen, M.~Rumpf, R.~Whitaker, Finite element methods on very
  large, dynamic tubular grid encoded implicit surfaces, SIAM J. Sci. Comput.

\bibitem{Dziuk:2008}
G.~Dziuk, C.~M. Elliott, {E}ulerian finite element method for parabolic {PDEs}
  on implicit surfaces, Interf. Free Bound. 10 (2008) 119--138.

\bibitem{LeungLowengrubZhao:2011:PDEs_solving_surf}
S.~Leung, J.~Lowengrub, H.~Zhao, A grid based particle method for solving
  partial differential equations on evolving surfaces and modeling high order
  geometrical motion, J. Comput. Phys. 230~(7) (2011) 2540--2561.
\newblock \href {http://dx.doi.org/10.1016/j.jcp.2010.12.029}
  {\path{doi:10.1016/j.jcp.2010.12.029}}.

\bibitem{Floater05:surface_param}
M.~Floater, K.~Hormann, Surface parameterization: a tutorial and survey, in:
  Advances in Multiresolution for Geometric Modelling, Springer, 2005, pp.
  157--186.

\bibitem{Ruuth/Merriman:jcp08:CPM}
S.~J. Ruuth, B.~Merriman, A simple embedding method for solving partial
  differential equations on surfaces, J. Comput. Phys. 227~(3) (2008)
  1943--1961.
\newblock \href {http://dx.doi.org/10.1016/j.jcp.2007.10.009}
  {\path{doi:10.1016/j.jcp.2007.10.009}}.

\bibitem{cbm:lscpm}
C.~B. Macdonald, S.~J. Ruuth, Level set equations on surfaces via the {C}losest
  {P}oint {M}ethod, J. Sci. Comput. 35~(2--3) (2008) 219--240.
\newblock \href {http://dx.doi.org/10.1007/s10915-008-9196-6}
  {\path{doi:10.1007/s10915-008-9196-6}}.

\bibitem{cbm:phd}
C.~B. Macdonald, The {C}losest {P}oint {M}ethod for time-dependent processes on
  surfaces, Ph.D. thesis, Simon Fraser University (August 2008).

\bibitem{cbm:icpm}
C.~B. Macdonald, S.~J. Ruuth, The implicit {C}losest {P}oint {M}ethod for the
  numerical solution of partial differential equations on surfaces, SIAM J.
  Sci. Comput. 31~(6) (2009) 4330--4350.
\newblock \href {http://dx.doi.org/10.1137/080740003}
  {\path{doi:10.1137/080740003}}.

\bibitem{Trefethen:barycentric04}
J.-P. Berrut, L.~N. Trefethen, Barycentric {L}agrange interpolation, SIAM Rev.
  46~(3) (2004) 501--517.

\bibitem{GolubvanLoan:matrix}
G.~H. Golub, C.~F. van Loan, Matrix Computations, Johns Hopkins University
  Press, 1996.

\bibitem{Trefethen:numlinalg:1997}
L.~N. Trefethen, D.~Bau, Numerical Linear Algebra, SIAM, 1997.

\bibitem{LAPACK}
E.~Anderson, Z.~Bai, C.~Bischof, S.~Blackford, J.~Demmel, J.~Dongarra,
  J.~Du~Croz, A.~Greenbaum, S.~Hammarling, A.~McKenney, D.~Sorensen, {LAPACK}
  Users' Guide, 3rd Edition, SIAM, 1999.

\bibitem{ARPACKUserGuide}
R.~Lehoucq, D.~Sorensen, C.~Yang, {ARPACK} Users' Guide: Solution of
  Large-Scale Eigenvalue Problems with Implicitly Restarted {A}rnoldi Methods,
  SIAM, 1998.

\bibitem{scipy}
E.~Jones, T.~Oliphant, P.~Peterson, et~al., {SciPy}: Open source scientific
  tools for {Python}, \url{http://www.scipy.org}, accessed 2008-07-15 (2001).

\bibitem{Oliphant:2007:numpy}
T.~E. Oliphant, {P}ython for scientific computing, Computing in Science {\&}
  Engineering 9~(3) (2007) 10--20.

\bibitem{vtk1998}
W.~Schroeder, K.~Martin, B.~Lorensen, The {V}isualization {T}oolkit: an
  object-oriented approach to {3D} graphics, Prentice Hall, 1998.

\bibitem{mayavi2}
P.~Ramachandran, G.~Varoquaux, Mayavi: Making {3D} data visualization reusable,
  in: G.~Varoquaux, T.~Vaught, J.~Millman (Eds.), Proceedings of the 7th Python
  in Science Conference, 2008, pp. 51--56,
  \url{http://code.enthought.com/projects/mayavi/}.

\bibitem{StanfordBunny}
G.~Turk, M.~Levoy, The {S}tanford {B}unny, the {S}tanford {3D} scanning
  repository, \url{http://www-graphics.stanford.edu/data/3Dscanrep}, accessed
  2008-06-18 (1994).

\bibitem{NASA:visible_earth}
{NASA}, {V}isible {E}arth website, \url{http://visibleearth.nasa.gov}, accessed
  2010-12-05 (2010).

\end{thebibliography}

\end{document}